\documentclass[11pt]{extarticle}

\usepackage[left=1in,right=1in,top=1in,bottom=1.3in]{geometry}
\usepackage{amsmath, amssymb, amsthm, latexsym, amsfonts, indentfirst, xcolor, mathtools, manfnt, microtype}
\usepackage[utf8]{inputenc}
\usepackage[english]{babel}

\numberwithin{equation}{section} % нумерация формул по разделам

\usepackage{indentfirst}
\usepackage{booktabs} 		% таблицы
\usepackage{ dsfont }

\usepackage{amsthm}

\theoremstyle{plain}
\newtheorem{thm}{Theorem}
\newtheorem{theorem}[thm]{Theorem}

\newtheorem{lem}[thm]{Lemma}

\newtheorem{obs}[thm]{Observation}

\theoremstyle{definition}

\usepackage{enumitem}

\newcommand{\FF}{\mathcal{F}}

\newcommand{\ff}{\mathcal{F}}

\newcommand{\m}{\mathcal}
\newcommand{\HH}{\mathcal{H}}
\usepackage{hyperref}

\title{Frankl's diversity theorem for permutations}
\date{}
\author{Eduard Inozemtsev\thanks{Moscow Institute of Physics and Technology;
E-mail: \url{eduard inozemtsev@bk.ru}}, Andrey Kupavskii\thanks{Moscow Institute of Physics and Technology, Saint-Petersburg State University, Innopolis University;
E-mail: \url{kupavskii@ya.ru}}}
\begin{document}
\maketitle
\begin{abstract}
    In 1987, Frankl proved an influential stability result for the Erd\H os--Ko--Rado theorem, which bounds the size of an intersecting family in terms of its distance from the nearest (subset of) star or trivial intersecting family. It is a far-reaching extension of the Hilton--Milner theorem.  In this paper, we prove its analogue for permutations on $\{1,\ldots, n\}$, provided $n$ is large. This provides a similar extension  of a Hilton--Milner type result for permutations proved by Ellis.
\end{abstract}
\section{Introduction}

We denote by $[n]:=\{ 1, \dots, n\}$ the standard $n$-element set  and by ${[n] \choose m} := \{ F \subset [n] : |F| = m \}$ the set of its $m$-element subsets. A {\it family} is a collection of sets. A family is {\it intersecting} if any two sets from the family intersect.

The famous theorem of Erd\H os, Ko and Rado  (EKR) \cite{EKR} states that for $n\ge 2m$ an intersecting family in ${[n]\choose m}$ has size at most ${n-1\choose m-1}.$ For $n>2m$ we have a series of stability results for the EKR theorem that can be conveniently formulated in terms of the diversity of a family. For a family $\FF \subset {2^{[n]}}$ we define its maximum degree $\Delta(\FF):=\max_{x \in [n]}|\FF[x]|$. Then {\it diversity} $\gamma(\FF)$ of the family $\FF$ is $\gamma(\FF)=|\FF|-\Delta(\FF).$ That is, diversity is the number of sets not containing the element of maximum degree. It can be thought of as the (edit) distance to the closest star. 

The result of Hilton and Milner \cite{HM} states that $|\m F|\le {n-1\choose m-1}-{n-m-1\choose m-1}+1$ if $\gamma(\m F)\ge 1.$ The matching example is the Frankl family $\m A_m$, where the Frankl families $\m A_k$ are defined for $k=3,\ldots, m$ below:
$$\m A_k = \Big\{F\in {[n]\choose m}:1\in F, F\cap [2,k+1]\ne\emptyset\Big\}\bigcup \Big\{F\in {[n]\choose m}: [2,k+1]\subset F\Big\}.$$
The following theorem was proved by Frankl \cite{Frankl87} in the degree form and then by Kupavskii and Zakharov \cite{KupZak18} in the diversity form.
\begin{thm}[Frankl, \cite{Frankl87}; Kupavskii and Zakharov, \cite{KupZak18}]\label{thmfrdiv}
    Let $n>2m>0$ be integers. Take an intersecting family $\FF\subset {[n]\choose m}$ with $\gamma(\FF)\geq{n-k-1 \choose m-k}$ for some real $3\leq k\leq m$. Then 
    $$|\FF|\leq{n-1 \choose m-1} + {n-k-1 \choose m-k} - {n-k-1 \choose m-1},$$
    and equality is possible only if $\FF$ is isomorphic to $\m A_k.$
\end{thm}

Intersection problems are studied for other objects: permutations, graphs \cite{EFF}, partitions \cite{MeMo}, \cite{Kup54}, simplicial complexes \cite{Bor3}, \cite{Kup55} and vector spaces \cite{FG}. In this paper, we focus on permutations. 
We denote by $\Sigma_n$ the family of all permutations $[n] \rightarrow [n]$. We say that two permutations $\sigma, \pi \in \Sigma_n$ \emph{intersect} if there exists $i \in [n]$ such that $\sigma(i) = \pi(i)$. It is convenient to identify a permutation $\sigma \in \Sigma_n$ with an $n$-element subset of $[n]^2$ consisting of pairs $(x, \sigma(x))$. This way, a family of permutations can be treated as a subfamily of ${[n]^2 \choose n}$. The intersection of two permutations then becomes the usual intersection of $n$-element sets.  

Deza and Frankl \cite{DeF} initiated the study of $t$-intersecting families of permutations. They observed  that if $\FF \subset \Sigma_n$ is intersecting, then $|\FF| \leq (n-1)!$. Let us recite the simple proof of this bound below. Take the standard $n$-cycle $\pi$, $\pi(k)=k+1$, with addition modulo $n$. Let $G$ be the cyclic group of order $n$ generated by $\pi$. Take a permutation $\sigma$. Consider its left coset $\sigma G$. Any two permutations in $\sigma G$ do not intersect. Hence, $\sigma G$ contains at most $1$ permutation from $\FF$. At the same time, we have $(n-1)!$ different cosets.

This proof is simple, but it does not say anything about the case when equality $|\FF| = (n-1)!$ is achieved. Deza and Frankl conjectured that $|\FF| = (n-1)!$ if and only if $\FF$ consists of all permutations with some fixed element $(x, y)$. Note that this family has diversity $0$.  It was proved by Cameron and Ku \cite{CK} and Larose and Malvenuto \cite{LM}.  Larose and Malvenuto conjectured that the largest intersecting family permutations $\m F$ with $\gamma(\m F)\ge 1$ is isomorphic to  
\begin{equation}\label{eqhmper}\{ \sigma \} \cup \big\{\pi\in \Sigma_n: \pi(1)=1, \pi \cap \sigma\ne \emptyset\big\},\end{equation}
where $\sigma\in \Sigma_n$ is some permutation with $\sigma(1)\neq1$.  By {\it isomorphic to $\FF$} we mean that a family has the form $\pi \FF \rho$, where $\pi, \rho \in \Sigma_n$. This Hilton--Milner type result was proved by Ellis \cite{Ell} for large $n$. The results mentioned in this paragraph were proved using algebraic techniques such as Hoffman spectral bounds and representation theory of symmetric groups.

 Let $\sigma \in \Sigma_n$ be a permutation. We denote by $\Sigma_{n, \sigma, \overline{\pi}}$ the family of permutations that intersect $\sigma$ and do not  intersect $\pi$.   If $\sigma$ is the identity permutation, then by $\m{D}_n := \Sigma_{n, \overline{\sigma}}$ we denote the family of derangements. Similarly, we use $\m{D}_{n, \sigma}$ to denote permutations from $\m{D}_{n}$ that intersect $\sigma$.
 We also denote $d_n = |\m{D}_n|$. 
 Using the inclusion-exclusion principle, it is easy to see that $|\m{D}_n| = n!\sum_{i=1}^n \frac{(-1)^{i+1}}{i!}$.  It is well-known  that $|\m{D}_n| = \lceil \frac{n!}{e} \rfloor$, i.e. $|\m{D}_n|$ is the nearest integer to $\frac{n!}{e}$.

In this paper, we get an analogue of Theorem~\ref{thmfrdiv} for permutations. First, we need to find the correct analogue of the Frankl families $\m A_k$. It will be given by $\m E_k$ below.

Recall that, for a given permutation $\sigma\in \Sigma_n$, its {\it support } $S(\sigma)$ is the subset of $[n]$ of all non-fixed points of $\sigma$. That is, $S(\sigma)=\{i: \sigma(i)\ne i\}.$ 
For $k=2,\ldots, n-2$ consider a family $$\HH_k := \Sigma_{n}[\overline{(1, 1)}, \{(k+1, k+1), \dots, (n, n)\}]$$ of all permutations $\sigma$ with $S(\sigma)\cap [k+1,n]=\emptyset$ and $1\notin S(\sigma)$. (Cf. the notation in the next section.) We have $$|\HH_k| = k! - (k-1)!.$$ For a given $\m{G}\subset \Sigma_n$ define $\m{N}(\m{G}) \subset \Sigma_n[(1, 1)]$ to be the family of permutations from $\Sigma_n[(1, 1)]$ that intersect each permutation from $\m{G}$. 
% all hitting sets of $\HH$ from $S_n[(1, 1)]$).
Define 
$$\m{E}_k := \HH_k \cup \m{N}(\HH_k).$$
It is easy to see that $\HH_k$ and $\m N(\m H_k)$ are disjoint and together form an intersecting family. Also, $\gamma(\m E_k) = |\m H_k| = k!-(k-1)!.$ Note that $\m E_2$ is isomorphic to the family \eqref{eqhmper}. One can  check that  $|\m{E}_{n-3}| = |\m{E}_{n-2}|=3(n-2)!-2(n-3)!$ using the inclusion-exclusion formula (see \eqref{eqinex}). It is also not difficult to see that $|\m{E}_k|-|\m{E}_{k-1}|\ge d_{n-2}$ for $3\le k\leq n-3$. \\

The main result of this paper is given below. It provides a generalization of a result by Ellis~\cite{Ell} analogous to Theorem~\ref{thmfrdiv}.  \begin{theorem} \label{diversity}
    Let $n$ be a large enough integer. Take an intersecting family $\FF\subset \Sigma_n$ of permutations. If for some $k\in [2,n-3]$ we have $\gamma(\FF) \ge k! - (k-1)!$ then $$|\FF| \leq |\m E_k| = (n-1)! - \sum_{i = 0}^{k-1} {k-1 \choose i} d_{n-1-i} + k! - (k-1)!$$ 
    Moreover, for $k\le n-4$ equality is only possible for $\m F$ isomorphic to $\m E_k$, and for $k=n-3$ equality is only possible when $\m F$ is isomorphic to either $\m E_{n-3}$ or $\m E_{n-2}$.
\end{theorem}
In the following section, we introduce some notation and give exact and  asymptotic formulas for the size of $\m E_k$. We note that, for any $k\in[2,n-3],$ it is dominated by the $|\m N(\m H_k)|$ term.   In Section~\ref{sec3} we prove Theorem~\ref{diversity}.

A related question that was studied for sets is  how large the diversity of an intersecting family can be. Lemons and Palmer \cite{LP} proved for $n>n_0(m)$ that if $\FF \subset {[n] \choose m}$ is intersecting, then $\gamma(\FF)\leq{n-3 \choose m-2}$. Frankl \cite{FAN} improved the bound on $n$ to $n \geq 6k^2$, then Kupavskii \cite{KD} improved it to $n \geq Ck$ for some absolute constant $C$.  Recently, Frankl and Wang \cite{FWD} proved that  for $n \geq 36k$, which is currently the best known bound. 

For permutations, Wang and Xiao \cite{WX}, using spread approximations, recently proved that if $\FF \subset \Sigma_n$ is intersecting, then $\gamma(\FF) \leq (n-2)!-(n-3)!.$ (We note that the authors of the present note, while working on the proof, also got the same result via the same technique. We, however, omit it and simply rely on \cite{WX} instead.)

Our proof combines several combinatorial ingredients with the spread approximation method. An interesting aspect of this application of the spread approximation technique is that we need its highly asymmetric variant for two cross-intersecting families, with a carefully chosen set of parameters. 

\section{The family $\m E_k$} \label{sec2}
Let $\FF$ be a family of sets. For sets $X, Y$ and a family $\m S$, we use the following standard notation:
\begin{align*}
\FF(X):=&\ \{F\setminus X:X\subset F, F\in \FF\},\\
\FF[X]:=&\ \{F: X\subset F, F\in \FF\},\\
\FF[X, \overline{Y}]:=&\ \{F: X\subset F, Y \cap F = \emptyset,  F\in \FF\} ,\\
\FF[\m{S}]:=&\ \bigcup_{A\in \m{S}}\FF[A].
\end{align*}
Recall that, whenever more convenient, we treat permutations $\sigma$ as sets in $[n]^2$. That is, we have $\sigma = \{(i,\sigma(i)): i\in[n]\}$. We will also work with {\it partial permutations,} i.e., sets of the same form and size $\le n$ that could be extended to a permutation. 

For convenience in this and the  following proofs,  we introduce a partial permutation $\alpha_{i,j} := \bigl(\begin{smallmatrix}
  i & i+1 &  \cdots &  j \\
  i & i+1 &  \cdots &  j
\end{smallmatrix}\bigr)=\{ (i,i),(i+1,i+1) \dots,(j,j) \}$  and use notation like $\Sigma_{\alpha_{k+1,n}}[(1,1)]$, which stands for all permutations from $\Sigma_n$ that intersect $\alpha_{k+1,n}$ and fix $(1,1).$ We will typically omit $n$   from the subscript in the notation like $\Sigma_n$ etc.,  simply writing $\Sigma$ instead.

We will need the following formula for the size of $\m N(\m H_k)$.
\begin{lem}\label{lemsizes} We have
\begin{align} 
    \label{eqsizehk} |\m{N}(\HH_k)|&= (n-1)! - 
    \sum_{i=0}^{k-1}{k-1 \choose i}d_{n-i-1}, \\
\label{eqsizeek}|\m{E}_k| &= (n-1)! - \sum_{i = 0}^{k-1} {k-1 \choose i} d_{n-1-i} + k! - (k-1)!.
\end{align}
\end{lem}
\begin{proof}
    We prove \eqref{eqsizehk}. Then \eqref{eqsizeek} follows.   It is easy to see that a permutation $\pi \in \Sigma[(1, 1)]$ intersects each permutation of $\HH_k$ if and only if $S(\pi)\cap [k+1,n]\ne \emptyset.$   In our notation, we have $\m{N}(\HH_k) = \Sigma_{\alpha_{k+1,n}}[(1, 1)]$. We decompose 
\begin{equation} \label{extrcount}
|\m{N}(\HH_k)| = |\Sigma_{\alpha_{k+1,n}}[(1, 1)]| 
 = |\Sigma[(1, 1)]| - |\Sigma_{\overline{\alpha_{k+1,n}}}[(1, 1)]|.
\end{equation} We find the size of $\Sigma_{ \overline{\alpha_{k+1,n}}}[(1, 1)]$ by specifying the intersection of its permutations with  $\alpha_{2, k}$:
$$
\Sigma_{\overline{\alpha_{k+1,n}}}[(1, 1)] = \bigcup_{i=0}^{k-1} \bigcup_{\pi \subset \alpha_{2,k}, |\pi| = i} \Sigma_{\pi, \overline{\alpha_{2,n}\setminus \pi}}[(1,1)].
$$
Note that for $|\pi|=i$  we have $|\Sigma_{\pi, \overline{\alpha_{2,n}\setminus \pi}}[(1,1)]|= d_{n-i-1}$. This gives
$$
|\Sigma_{\overline{\alpha_{k+1,n}}}[(1, 1)]| = \sum_{i=0}^{k-1}{k-1 \choose i}d_{n-i-1}.
$$
Together with \eqref{extrcount} this implies \eqref{eqsizehk}.
\end{proof}

\begin{lem}\label{lemsizehk}
For $n\to \infty$ we have
\begin{align}\label{eqsizeekbound1}|\m N(\m H_k)|&=(n-2)!\cdot \Big((n-k)+O\big((n-k)^2/n\big)\Big).\\
\label{eqsizeekbound2}
|\m N(\m H_k)|&=(n-1)!\cdot\Big(1-e^{-1} \Big(1+\frac 1n\Big)^{k-1}+O(k\log^2 n/n^2)\Big).
\end{align}
\end{lem}
The first equality is applicable for $k=(1-o(1))n,$ while the second one is applicable for, say, $k<n-n^{1/2}$. For $k=cn$ with $c\in [0,1)$ it gives $|\m N(\m H_k)|=(n-1)!\cdot (1-e^{c-1}+o(1))$.
\begin{proof}
    To get \eqref{eqsizeekbound1}, we use the representation 
    \begin{equation}\label{eqinex}|\m N(\m H_k)| = \sum_{i=1}^{n-k}(-1)^{i+1}{n-k\choose i}(n-1-i)!,\end{equation}
    which could be obtained via exclusion-inclusion in the same vein as the formula for derangements. More precisely, it is implied from the following calculations
    \begin{align*}
    |\m{N}(\m H_k)| &= \sum_{i=1}^{n-k}(-1)^{i+1}{n-k\choose i}(n-1-i)! \\&= {n-k\choose 1}(n-2)!-{n-k\choose 2}(n-3)!+\sum_{i=3}^{n-k}(-1)^{i+1}{n-k\choose i}(n-1-i)!\\
    &= \Big( (n-k)(n-2)! + O\Big((n-k)^2/2\Big)\cdot(n-3)! + O\Big((n-k)^3/6\Big)\cdot(n-4)!\Big),
    \end{align*}
    where in the last equality we used 
    $$
    {n-k\choose i+1}/{n-k\choose i}\cdot(n-2-i)!/(n-1-i)! \leq 1/(i+1).
    $$
    
     In order to obtain \eqref{eqsizeekbound2}, we use \eqref{eqsizehk} and replace $d_{n-1-i}$ by $(n-1-i)!/e$:
    % \begin{align*}
    % |\m H_k| &= (n-1)! - \frac 1e\sum_{i = 0}^{k-1} {k-1 \choose i} (n-i-1)! +O(2^k)\\
    % &= (n-1)!\cdot\Big[1-\frac 1e\sum_{i = 0}^{k-1} {k-1 \choose i} n^{-i}\cdot \Big(1+O\big(i^2/n\big)\Big)\Big].
    % \end{align*}
    % Note that the terms in the sum decay at least as  $\frac 1{i!}$, and thus for the asymptotic as in \eqref{eqsizeekbound2} all but the  the first $10\log n$ terms of the sum fall into the $O(n)$ part. For the first $10\log n$ terms, $i^2/n$ is $O(\log^2n/n)$, and for term $i=0$ there is no error. It means that we could continue the above as follows.
    %  \begin{align*}
    % |\m H_k| &= (n-1)!\cdot\Big[1-\frac 1e\sum_{i = 0}^{k-1} {k-1 \choose i} n^{-i}\cdot \Big(1+O\big(i^2/n\big)\Big)\Big]\\
    % &= (n-1)!\cdot\Big[1-\frac 1e\sum_{i = 0}^{k-1} {k-1 \choose i} n^{-i}+O(k\log^2 n/n^2)\Big]\\
    % &= (n-1)!\cdot\Big[1-\frac 1e\Big(1+\frac 1n\Big)^{k-1}+O(k\log^2 n/n^2)\Big] 
    % \end{align*}

    \begin{align*}
    |\m H_k| &= (n-1)! - \frac 1e\sum_{i = 0}^{k-1} {k-1 \choose i} (n-i-1)! +O(2^k)\\
    &= (n-1)!\left( 1 - \frac 1e\sum_{i = 0}^{k-1} {k-1 \choose i} \frac{(n-i-1)!}{(n-1)!} + O\Big(\big( 2e/n \big)^n\Big)  \right).
    \end{align*}

    Next we estimate the sum $\sum_{i = 0}^{k-1} {k-1 \choose i} \frac{(n-i-1)!}{(n-1)!}$ in the following way
    \begin{align*}
    & \sum_{i = 0}^{k-1} {k-1 \choose i} \frac{(n-i-1)!}{(n-1)!} = \sum_{i = 0}^{10\log{n}} {k-1 \choose i} \frac{(n-i-1)!}{(n-1)!} + \sum_{i = 10\log{n}+1}^{k-1} {k-1 \choose i} \frac{(n-i-1)!}{(n-1)!}\\
     & = \sum_{i = 0}^{10\log{n}} {k-1 \choose i} n^{-i}\Big( 1+O\Big(i^2/n\Big) \Big) + O\left( {k-1 \choose 10\log{n}}\frac{(n-10\log{n}-1)!}{(n-1)!} \right) \\
     & = \sum_{i = 0}^{10\log{n}} {k-1 \choose i} n^{-i}\Big( 1+O\Big(i^2/n\Big) \Big) + O\left( \Big(\frac{(k-1)e}{10\log{n}}\Big)^{10\log{n}}\cdot \Big(\frac{1}{n-10\log{n}}\Big)^{10\log{n}} \right)  \\ 
     & = \sum_{i = 0}^{10\log{n}} {k-1 \choose i} n^{-i}\Big( 1+O\left(i^2/n\Big) \right) + O\left(1/n^C\right)
    \end{align*}
    for any constant $C$, where in the first sum we used $\frac{(n-i)!}{n!} = n^{-i}\Big( 1 + O(i^2/n) \Big)$ for $i=o(\sqrt{n})$ and the second sum decreases at least as the geometrical progression with common ratio $\frac{1}{10\log{n}}$, since 
    $$
    {k-1 \choose i+1} \frac{(n-i-2)!}{(n-1)!} / {k-1 \choose i} \frac{(n-i-1)!}{(n-1)!} = \frac{k-1-i}{(i+1)(n-i-1)}\leq \frac{1}{10\log{n}}.
    $$

    We continue the above as follows, noting that for term $i=0$ there is no error, 
    \begin{align*}
    &  \sum_{i = 0}^{10\log{n}} {k-1 \choose i} n^{-i}\Big( 1+O\Big(i^2/n\Big) \Big) = \sum_{i = 0}^{10\log{n}} {k-1 \choose i} n^{-i}+\sum_{i = 1}^{10\log{n}} {k-1 \choose i} n^{-i}\cdot O\Big(i^2/n\Big) \\
    & = \sum_{i = 0}^{10\log{n}} {k-1 \choose i} n^{-i}+O(\log^2{n}/n)\cdot O(k/n) = \sum_{i = 0}^{k-1} {k-1 \choose i} n^{-i} + O\Big(k\log^2{n}/n^2\Big),
    \end{align*}
    where in the second equality we used $i\leq 10\log{n}$ and $\sum_{i = 1}^{10\log{n}} {k-1 \choose i} n^{-i} = O(k/n),$ and in the last inequality we again used that $\sum_{i = 10\log{n}+1}^{k-1} {k-1 \choose i} n^{-i} = O(1/n^C).$

    Finally, we have 
    \begin{align*}
    |\m H_k| &=  (n-1)!\cdot\Big[1-\frac 1e\sum_{i = 0}^{k-1} {k-1 \choose i} n^{-i}+O(k\log^2 n/n^2)\Big]\\
    &= (n-1)!\cdot\Big[1-\frac 1e\Big(1+\frac 1n\Big)^{k-1}+O(k\log^2 n/n^2)\Big]. 
    \end{align*}
\end{proof}

From here after, we replace $d_{n}$ by $n!/e$, since 
$$
\sum_{i=0}^m {m \choose i}d_{n-i} = \sum_{i=0}^m {m \choose i}(n-i)!/e + O(2^m)
$$
and the big-O term is negligible.

\section{Proof of Theorem \ref{diversity}}\label{sec3}

Permutations from $\Sigma_n[\overline{(1, 1)}]$ can be viewed as $(n-2)$-partial permutations $[2,n] \to [2,n]$: for $\sigma \in \Sigma_n[\overline{(1, 1)}]$,  delete $(1, \sigma(1))$ and $(\sigma^{-1}(1), 1),$ and get an $(n-2)-$partial permutation. Conversely, such an $(n-2)$-partial permutation can be uniquely extended to a permutation in $\Sigma_n[\overline{(1, 1)}]$.  Note that $$\Big|\bigcap \HH_k\Big| = n-k$$ as permutations and also as partial permutations. (Here and in what follows, $\bigcap \m X$ is a shorthand for $\bigcap_{X\in \m X}X$.) Note, that we want to maximize $\m{N}(\m{G})$ for given $\m{G}$. Therefore, if $X = \bigcap \HH_k$ contains elements $(1, i)$ or $(i, 1)$, then permutations from $\Sigma[(1,1)]$ can not intersect $X$ by these elements. Hence, w.l.o.g. we may assume that $\bigcap \HH_k \subset [2, n]^2$.

Let $\FF \subset \Sigma_n=\Sigma$ be the largest intersecting family of permutations satisfying the conditions of Theorem~\ref{diversity}. W.l.o.g. assume that $(1, 1)$ has the largest degree in $\FF$: i.e., that $|\FF[(1,1)]| = \Delta(\FF)$. We decompose  $$\FF = \FF_\Delta  \sqcup \FF_\gamma,$$   where $\FF_\Delta = \FF[(1, 1)] $ is the degree part and $\FF_\gamma = \FF[\overline{(1, 1)}]$ is the diversity part. 
%Since we treat $\Sigma_n[\overline{(1, 1)}]$ as $(n-2)$-partial permutations $[2,n] \to [2,n]$, we have  $\bigcap \FF_\gamma \subset [2,n]^2.$ We also have 
We have $|\FF_\gamma|= \gamma(\FF) \geq k! - (k-1)!$ by the assumption of the theorem.  However, a bulk of the family is $\ff_\Delta.$ In the proof, we use the aforementioned result of Wang and Xiao \cite{WX} that for $n$ large enough and any intersecting $\FF \subset \Sigma_n$ we have  $$|\FF_\gamma|\leq(n-2)!-(n-3)!$$ Compared with the bounds from Lemma \ref{lemsizehk} on the size of $\m E_k$, it is typically a negligible part.

Put $X:=\bigcap \FF_\gamma$  and suppose that $|X| = n-t$, for some $t$.  That is, using $|\FF_\gamma| \geq k! - (k-1)!$, we have 
$$\Big|\bigcap \FF_\gamma\Big| = n - t \leq n - k.$$ If $t=k$, equivalent to $|X|=n - k$, then we have $|\Sigma[\overline{(1,1)},X]|= k! - (k-1)!$. We have $\ff_\gamma\subset \Sigma[\overline{(1,1)},X]$. The lower bound on the size of $\ff_\gamma$ implies that  $\ff_\gamma = \Sigma[\overline{(1,1)},X],$ which in turn is  isomorphic to $\m H_k.$ This and the fact that $\ff$ is inclusion-maximal implies that $\ff$ is isomorphic to $\m E_k.$

In what follows, we assume that $|X|= n - t < n - k$. For simplicity and w.l.o.g., in what follows we assume that $X = \alpha_{t+1, n}$. The equality $\bigcap \FF_\gamma = \alpha_{t+1, n}$ implies that for any $x \in [2,t]^2$ there exists a partial permutation $\pi$ from $\FF_\gamma$ such that $x \notin \pi$.

Our goal is to show that for $t>k$ we have $|\FF| < |\m{E}_k|$.  We distinguish three cases.

    \subsection{\textbf{Case 1.} $t \geq 4k$}
    We begin with the following lemma.
\begin{lem} \label{doubleder}
    Let $\sigma, \pi$ be two  
    partial permutations on $[n]$, and suppose that 
    $(x_1, y_1), (x_2, y_2) \notin \sigma, \pi$ and $\{(x_1,y_1),(x_2,y_2)\}$ is a partial permutation of size $2$. 
    Then the number of permutations avoiding both $\sigma$ and $\pi$, but containing $(x_1,y_1)$ and $(x_2,y_2)$ satisfies
    $$|\Sigma_{\overline{\sigma}, \overline{\pi}}[(x_1, y_1), (x_2, y_2)]| \geq \frac{(n-2)!}{e^2} - \frac{2(n-3)!}{e^2}.$$ %Also, if $\sigma$ and $\pi$ are two partial permutations, then the same conclusion holds. 
\end{lem}
\begin{proof}
    For simplicity, we assume that $(x_1,y_1) = (1,1)$ and $(x_2,y_2) = (2,2)$. We lower bound  $|\Sigma_{\overline{\sigma}, \overline{\pi}}[(1,1), (2, 2)]|$ using the permanent of the corresponding matrix. First, we form disjoint partial permutations $\sigma',\pi'$, where $\sigma'=\sigma\cap ([3,n]\times [3,n])$ and $ \pi'=(\pi\setminus\sigma)\cap ([3,n]\times [3,n])$ is defined analogously. Consider the $(0,1)$-matrix $A$ with rows and columns indexed by $[3,n]$ and with $0$'s exactly on positions corresponding to the elements of  $\sigma',\pi'$. This matrix has at most two zeros in each row and each column. Henderson proved \cite{He} that the minimum possible permanent of such matrix equals to the $(n-2)$-th menage number $U_{n-2}$ for $n$ even (and $U_{n-2} - 1$ for $n$ odd). We have  $$U_n \sim \frac{n!}{e^2}\Big(1 - \frac{1}{n-1}+\frac{1}{2!(n-1)(n-2)}-\dots\Big)\ge \frac{n!}{e^2} \cdot\frac{n-2}{n-1}  $$ for large enough $n$,  see \cite{MN}.

\end{proof}

%    $ \begin{aligned}
 %       |\Sigma_{n, \overline{\sigma}, \overline{\pi}}[(x_1, y_1), (x_2, y_2)]|  \geq & (n-4)^{n-2}\frac{(n-2)!}{(n-2)^{n-2}}  = (n-2)!\left(1 - \frac{2}{n-2}\right)^{n-2} \\ \geq (n-2)!e^{(n-2)\ln{\left(1 - \frac{2}{n-2}\right)}} & \geq (n-2)!e^{(n-2)(-\frac{2}{n-2}-\frac{4}{(n-2)^2})} \geq (n-2)!e^{-\frac{4}{n-2}} \\ \geq \frac{(n-2)!}{e^2}\left(1 - \frac{4}{n-2}\right).
 %   \end{aligned}$
    
 %   (The same bound may be derived by noting that $|\Sigma_{n, \overline{\sigma}, \overline{\pi}}|$ equals to the permanent of $(0,1)$-matrix with exactly two zeros in every row and column, which equals to $n$ -th menage number $U_n$.  And the asymptotics of $U_n$ is $U_n \sim \frac{n!}{e^2}(1 - \frac{1}{n-1}+\frac{1}{2!(n-1)(n-2)}-\dots)$ \cite{MN}.) 

%    The second statement holds because we can extend partial permutations to full permutations.

We want to lower bound the size of $\Sigma[(1, 1)]\setminus \ff_{\Delta}$. Take any $\sigma \in \FF_\gamma$. Obviously, $\ff_{\Delta}\cap \Sigma_{\overline{\sigma}}[(1,1)]=\emptyset.$ But we could say more. 
For any $x  \in \sigma$ there exists $\rho_x \in \FF_\gamma$ such that $x \notin \rho_x$, because $\bigcap \FF_\gamma = \alpha_{t+1, n}$ by our assumption. The family $\ff_\Delta$ cannot contain permutations from $\Sigma_{\overline{\sigma \setminus x}, \overline{\rho_x}}[(1, 1), x]$, since they do not intersect $\rho_x$. %Notice that $\rho_x \setminus \sigma$ and $\sigma \setminus x$ are disjoint partial permutations. 
By Lemma \ref{doubleder} we have $$|\Sigma_{\overline{\sigma\setminus x}, \overline{\rho_x}}[(1, 1), x]| \geq \frac{(n-2)!}{e^2} - \frac{2(n-3)!}{e^2}.$$
Note, that $\Sigma_{\overline{\sigma\setminus x}, \overline{\rho_x}}[(1, 1), x] \cap \Sigma_{\overline{\sigma\setminus y}, \overline{\rho_y}}[(1, 1), y] = \emptyset$. Therefore, summing up over all $x \in \sigma \cap [2, t]^2$, we have
\begin{align}
\notag |\Sigma[(1,1)]|-|\ff_\Delta| &\geq  |\m D_{\sigma}[(1,1)]| +(t-2)\frac{(n-2)!}{e^2} - (t-2)\frac{4(n-3)!}{e^2}\\
\label{eqcompare1}&\geq \frac {(n-1)!}e \Big(1+ \frac {t}{en}+O(t/ n^{2})\Big). 
\end{align}
Using \eqref{eqsizeekbound2}, we bound 
\begin{align}
\notag|\Sigma[(1,1)]|-|\m N(\m H_k)|&= (n-1)!\cdot \Big(e^{-1}\Big(1+\frac 1n\Big)^{k-1}+O(k\log^2 n/n^2)\Big) \\
\label{eqcompare2}&\le \frac{(n-1)!}e\Big(e^{k/n}+O(k\log^2 n/n^2)\Big)
\end{align}
Taking the difference of the RHS of \eqref{eqcompare1} and that of RHS \eqref{eqcompare2}, we get $$|\m N(\m H_k)|-|\m F_\Delta|\ge \frac{(n-1)!}e \Big(1+\frac t{en}-e^{k/n}+O(k\log^2 n/n^2+t/n^2)\Big).$$
We have $t \geq 4k$, and so $k \leq \frac{n}{4}$. In particular, $e^{k/n}\le 1+1.2 \frac kn\le 1+\frac t{3n}$ in this range. Thus, we get $$|\m N(\m H_k)|-|\m F_\Delta|\ge (n-1)!\cdot  \frac t{C n}$$
for some absolute $C>0$ and for $n$ large enough.
Consequently, 
$$
|\m{E}_k|-|\FF| \ge |\m{N}(\HH_k)|  - |\FF_\Delta|- |\FF_\gamma| \ge (n-1)!\cdot  \frac t{C n} - |\FF_\gamma|.
$$
If $t>4C$ then the last expression is strictly positive since  $|\FF_\gamma| \leq (n-2)!-(n-3)!$ from \cite{WX}. The theorem follows. If $t<4C$ then $|\ff_\gamma|\le t! = o((n-2)!)$ and we are done again.

\subsection{\textbf{Case 2.} $t\leq 4k, k \leq \frac{\sqrt{n}}{100}$} 

     Recall that $\bigcap \FF_\gamma = \alpha_{t+1, n}$ and $\bigcap \HH_k = \alpha_{k+1, n}$. Thus, both $\ff_\Delta$ and $\m N(\m H_k)$ contain   $\Sigma_{\alpha_{t+1,n}}[(1,1)] $.  Hence, we compare the remainders, $|\m{N}(\HH)\cap \Sigma_{\overline{\alpha_{t+1,n}}}[(1, 1)]|=|\Sigma_{\alpha_{k+1,t}, \overline{\alpha_{t+1,n}}}[(1, 1)]|$ and $|\FF_\Delta \cap \Sigma_{ \overline{\alpha_{t+1,n}}}[(1, 1)]|.$

In order to upper bound $|\FF_\Delta \cap \Sigma_{ \overline{\alpha_{t+1,n}}}[(1, 1)]|$, we count the number of hitting sets of different sizes for the family of partial permutations $\FF_\gamma(\alpha_{t+1, n})$. Recall that  $\FF_\gamma(\alpha_{t+1, n}) = \{ \pi \setminus \alpha_{t+1, n} \mid \pi \in  \FF_\gamma\}$ and that a {\it hitting set} for $\m X$ is a set that intersects all sets from $\m X$.  

Note that each permutation in  $\FF_\Delta \cap \Sigma_{ \overline{\alpha_{t+1,n}}}[(1, 1)]$ corresponds to a hitting set of size at most $t-1$ for $\FF_\gamma(\alpha_{t+1, n})$. We count inclusion-minimal hitting sets only. If $\m B$ is the  family of all such inclusion-minimal hitting sets, then, by maximality of $\ff,$
$$\FF_\Delta\cap \Sigma_{ \overline{\alpha_{t+1,n}}}[(1, 1)]= \Sigma_{ \overline{\alpha_{t+1,n}}}[(1,1), \m B].$$
%of size $i$ (we call them $i-$hitting sets) from $[2,t]^2$ for $i\in[2,t-2]$. Note that, since $\bigcap \FF_\gamma = \alpha_{t+1, n}$, for each element $x \in [2,t]^2$ there is a permutation that does not contain $x$, so there are no $1$-hitting sets.  If $\m{A}$ is a family of all such hitting sets, then $|\m{N}(\FF_\gamma) \cap \Sigma_{n, \overline{\alpha_{t+1,n}}}[(1, 1)]| \leq |\Sigma_{n}[(1,1), \m{A}]|$ But we need only \emph{minimal} hitting sets (hitting set is called minimal, if it does not contain hitting set as a subset). That is because if $X \subset Y$ are two hitting sets, then $\Sigma_{n}[(1,1), Y] \subset \Sigma_{n}[(1,1), X]$. 
The following lemma in the spirit of Erd\H os and Lov\`asz provides the necessary bound.  
\begin{lem} \label{hitsets}
    The family  $\FF_{\gamma}(\alpha_{t+1,n})$ has no $1$-hitting sets. The number of minimal $i$-hitting sets from $[2,t]^2$  for the family $\FF_{\gamma}(\alpha_{t+1,n})$ is at most $(t-2)^i$ for each $i \in[2,t-1]$.
\end{lem}
\begin{proof}
    Take  an arbitrary $\rho_0\in \FF_\gamma(\alpha_{t+1, n})$. Any hitting set $H$ contains an element from $\rho_0, $ say $x_1$. Take a permutation from $\FF_\gamma(\alpha_{t+1, n})$ that avoids $x_1$. We see that, first, there are no $1-$hitting sets and, second, any $(\ge 2)$-hitting set contains an element from $\rho_1$.  Given that $|\rho_0|,|\rho_1| \leq t-2$ (as a partial permutation on $[2,t]$),  the number of $2-$hitting sets is at most $(t-2)^2$. 
    
    We proceed inductively. For each prospective partial hitting set $x_1,\ldots, x_i$ that is not yet a hitting set, there is a permutation $\rho_i$ that avoids all of $x_1,\ldots, x_i$.  Any inclusion-minimal $(i+1)-$hitting set $H$ extending $x_1,\ldots, x_i$ must contain an element $x_{i+1} \in \rho_i$. By induction, the number of partial sets $x_1,\ldots, x_i$ is at most $(t-2)^i,$ and there are $(t-2)$ options for $x_{i+1}.$ Moreover, each inclusion-minimal $(i+1)$-hitting set must arise this way. Thus, there are at most $(t-2)^{i+1}$ inclusion-minimal $(i+1)$-hitting sets. 
\end{proof}
Define $\m B$ to be the family of such hitting sets. Note that the number of sets from $\Sigma_{ \overline{\alpha_{t+1,n}}}[(1,1)]$ containing an $i$-hitting set $B\in \m B$ is equal to 
$$\sum_{j=i}^{t-i-1}{t-i-1\choose j-i}d_{n-j-1},$$
by running the same argument as in Lemma~\ref{lemsizes}. This and the previous  lemma imply
\begin{align} 
\notag|\FF_\Delta \cap \Sigma_{\overline{\alpha_{t+1,n}}}[(1, 1)]| &\leq |\Sigma_{\overline{\alpha_{t+1,n}}}[(1, 1),\m B]|\\
\notag &\le \sum_{i=2}^{t-1} (t-2)^i \sum_{j=i}^{t-i-1}{t-i-1\choose j-i}d_{n-j-1} \notag \\
\notag &\le e^{-1}\sum_{i=2}^{t-1} \sum_{j=i}^{t-i-1}t^j (n-j-1)! \notag \\
&\le t^2\frac{(n-3)!}{e} + 2t^3(n-4)!+3t^4(n-5)!+\dots \notag \\
&\leq t^2\frac{(n-3)!}{e} + 2t^2/25 \cdot \sqrt{n}(n-4)!+3t^2/25^2 \cdot(\sqrt{n})^2 (n-5)!+\dots  \notag \\
\label{eqcompare3} &\leq t^2\frac{(n-3)!}{e}+ t^2 \sqrt{n}(n-4)!,
\end{align}
where we used $t \leq 4k \leq \frac{\sqrt{n}}{25}.$

At the same time,  we bound $|\m{N}(\HH_k)\cap \Sigma_{ \overline{\alpha_{t+1,n}}}[(1, 1)]|=|\Sigma_{ \alpha_{k+1,t}, \overline{\alpha_{t+1,n}}}[(1, 1)]|$ as follows
    \begin{align} \label{hitsetsbounds}
     \notag   |\m{N}(\HH_k)\cap \Sigma_{ \overline{\alpha_{t+1,n}}}[(1, 1)]| &= |\Sigma_{\overline{\alpha_{t+1,n}}}[(1, 1)]| - |\Sigma_{\overline{\alpha_{k+1,n}}}[(1, 1)]|\\
     &=   |\m{N}(\HH_t)|- |\m{N}(\HH_k)|\\
     \notag &\overset{\eqref{eqsizehk}}{=} \sum_{i = 0}^{t-1} {t-1 \choose i} d_{n-i-1} - \sum_{i = 0}^{k-1} {k-1 \choose i} d_{n-i-1} \\
     \label{Hbound}    &\geq (t-k)\frac{(n-2)!}{e}.
    \end{align}
%where we used $||S_{n, \overline{\alpha_{t+1,n}}}[(1, 1)]|$ and $|S_{n, \overline{\alpha_{k+1,n}}}[(1, 1)]|$ as in Lemma~\ref{lemsizes}.

Finally, compare the expressions in the RHS of \eqref{Hbound} and \eqref{eqcompare3}. %We have $(t-k)\frac{(n-2)!}{e} > t^2\frac{(n-3)!}{e}+ t^2 \sqrt{n}(n-4)!$ for $t\le \sqrt n/25$. from the inequality above.
We have 
$$
(t-k)\frac{(n-2)!}{e}-t^2\frac{(n-3)!}{e}-t^2 \sqrt{n}(n-4)! \ge \frac{(n-2)!}{e}\left(t-k-\frac{2t^2}{n-2}\right)\geq \frac{(n-2)!}{e}\left(1-\frac{1}{10}\right).
$$
where we used that $t\le \sqrt n/25$ and $k < t$. It is clear that the RHS of the displayed chain of inequalities is bigger than $|\ff_\gamma|,$ since in the case $t \leq \frac{\sqrt{n}}{25}$ we have $|\FF_\gamma|\leq (n-t)!\leq \Big(n - \frac{\sqrt{n}}{25}\Big)!$. Thus we conclude that $|\ff|<|\m E_k|$ again.

%%%%%%%%%%%%%%%%%%%%%%%%%%%%%%

\subsection{\textbf{Case 3.} $k\ge \sqrt n/100$, $t \leq 4k$}
As in Case 2, we compare the sizes of $\m{N}(\HH_k)\cap \Sigma_{\overline{\alpha_{t+1,n}}}[(1, 1)]=\Sigma_{\alpha_{k+1,t}, \overline{\alpha_{t+1,n}}}[(1, 1)]$ and $\FF_\Delta \cap \Sigma_{ \overline{\alpha_{t+1,n}}}[(1, 1)].$ 
From now on, we treat permutations from $\FF_{\gamma}$ as $(t-2)-$partial permutations $[2,t] \rightarrow [2,t]$, and permutations from $\FF_\Delta$ as permutations $[2, n] \rightarrow [2, n]$. 

To prove the theorem in this case, we shall use the spread approximation technique, introduced by the second author and Zakharov \cite{KuZa} and developed in a series of recent works of the second author and coauthors \cite{Kupavskii26,
Kupavskii24Perm,
FranklKup25,
IarKup25,
KupNos25}.

We say that a family $\FF \subset  {[n] \choose k}$ is $r$-spread for some $r\geq 1$ if $\frac{|\FF(X)|}{|\FF|} \leq r^{-|X|}$ for any $X \subset [n].$ The following statement is a variant due to Stoeckl \cite{St}   of  the breakthrough result that was proved  by Alweiss, Lovett, Wu and Zhang \cite{Alw}. It is central for the spread approximation method.

\begin{thm} [Spread Lemma, \cite{Alw}] \label{spreadtheorem}
    If for some $n, k, r, m \geq 1$ and $\delta>0$ a family $\FF \subset {[n] \choose \leq k}$ is $r$-spread and $W$ is an $ m\delta$--random subset of $[n]$, then 
    $$
        \Pr[\exists F \in \FF : F \subset W] \geq 1 - \left(\frac{2}{\log_2(r\delta)}\right )^m k.
    $$
\end{thm}

For us, it will be convenient to work with the following modified version of spreadness. Fix an
arbitrary family of sets $\m{A}$ and a real number $\tau>1$. A subfamily $\FF \subset \m{A}$ is called $(\m{A}, \tau)$-homogeneous if for any set $S$ we have 
$$
\frac{|\FF(S)|}{|\FF|}\leq \tau^{|S|}\frac{|\m{A}(S)|}{|\m{A}|}.
$$ 

\begin{obs}\label{obs3} Given a family $\m{A}$, $\tau>1$ and a family $\ff \subset \m{A}$, let $X$ be a maximal set that satisfies $|\ff(X)|\ge\tau^{|X|} \frac{|\m{A}(X)|}{|\m{A}|}|\ff|$. Then $\ff(X)$ is $(\m{A}(X),\tau)$-homogeneous.
\end{obs}
\begin{proof}
Indeed, for any $B\supsetneq X$ of size $b$ we have
$$|\ff(B)|\le \tau^b \frac{(n-b)!}{n!}|\ff|\le\tau^{b-|X|} \frac{(n-b)!}{n!} |\ff(X)|.$$
\end{proof}

Also we need the following notion.
Given a family $\m{A}\subset 2^{[n]}$ of sets and $q,r\ge 1$, we say that $\m{A}$ is {\it $(r,q)$-spread} if for each $S\in {[n]\choose \le q}$, the family $\m{A}(S)$ is $r$-spread. Note that putting $S = \emptyset$ implies that $\m{A}$ is $r$-spread, so $(r, q)$-spreadness is a stronger condition than the usual $r$-spreadness.

For any partial permutations $S,X$, such that $S\subset X$, we have $\frac{|\Sigma_n(X)|}{|\Sigma_n(S)|} = \frac {(n-|X|)!}{(n-|S|)!}$, and simple calculus shows that for any $|S|<n/4$ and any $X\supset S$, where $X,S$ are partial permutations, we have
$$\frac{|\Sigma_n(S)|}{|\Sigma_n(X)|} = \frac {(n-|S|)!}{(n-|X|)!}\ge \big((n-|S|)!\big)^{\frac{|X|-|S|}{n-|S|}}\ge \Big(\frac{n-|S|}{e}\Big)^{|X|-|S|}>\Big(\frac n4\Big)^{|X|-|S|}.$$
That is, $\Sigma_n$ is $(\frac{n}{4}, \frac{n}{4})$-spread.

We will apply the following spread approximation result for $\ff_\gamma, \FF_\Delta$. One of the unusual aspects of this result is a highly asymmetric choice of parameters for the two families.  In what follows, we write $\log{n}:=\log_{1.01}{n}$ for shorthand, we put $n_1 := n-1, n_2 := t-1$ and omit integer parts whenever not essential to simplify presentation. 
\begin{thm}\label{thmapprox}
    Consider a family $\FF_1$  of permutations on $[2,n]$ and a family $\ff_2$ of partial permutations on $[2,t]$, such that $\FF_1$ and $\FF_2$ are cross-intersecting. 
    Let $\tau_1 := 1.01, \tau_2:=\frac{k}{\log^2{n}}$ and let $q_1:=10\log{n}, q_2:=\lceil1.1(t-k)\rceil.$ Then there exist families $\m{S}_1, \m S_2$ of partial permutations on  $[2,n]$ and $[2,t]$ of sizes at most $q_1,q_2$, respectively, and families $\FF_i' \subset \FF_i$ for $i\in\{1,2\}$ such that the following holds.
    \begin{enumerate}[label=(\roman*)]
        \item We have $\FF_i \setminus \FF_i'\subset \Sigma_{n_i}[\m S_i]$;
        \item for any $B \in \m{S}_i$ there is a family $\FF_{i, B} \subset \FF_i$ such that $\FF_{i,B}(B)$ is $(\Sigma_{n_i}(B), \tau_i)$-homogeneous;
        \item $|\FF_i'| \leq \tau_i^{-q_i-1}n_i!$.
    \end{enumerate}
\end{thm}
\begin{proof} The first guiding observation to make is that $\FF$ cannot be $\tau$-homogeneous.  
Indeed, the same application of Spread Lemma as we will use later imply that there are $A_1 \in \FF_1, A_2 \in \FF_2$ such that $A_1 \cap A_2=\emptyset$. We avoid calculations here, since we do them below in a more general scenario.  

\begin{enumerate}
    \item Find an inclusion-maximal $S_i^{j}$ that  $|\FF_{i}^j(S_i^{j})| \ge  \tau^{|S_i^{j}|}_i \frac{|\Sigma_{n_i}(S)|}{|\Sigma_{n_i}|}|\FF_i^{j}|$.
    \item If $|S_i^{j}|> q_i$ or $\FF_{i}^j = \emptyset$ then stop. Otherwise, put $\FF_{i}^{j+1}:=\FF_{i}^j\setminus \FF_{i}^{j}[S_i^{j}]$. 
\end{enumerate}
Let us show that $\FF_{i}^j(S_i^{j})$ is $(\Sigma_{n_i}(B), \tau_i)$-homogeneous. Indeed, for any set $X$ disjoint with $S_i^j$ we have
$$|\FF_{i}^j(S_i^{j}\cup X)|< \tau_i^{|S_i^{j}|+|X|} \frac{|\Sigma_{n_i}(S_i^j\cup X)|}{|\Sigma_{n_i}(X)|} |\ff_i^j|\le \tau_i^{|X|} \frac{|\Sigma_{n_i}(S_i^j)|}{|\Sigma_{n_i}|} |\FF_{i}^j (S_i^{j})|,$$
where the first inequality uses the maximality of $S_i^j.$

Let $N_i$ be the step of the procedure for $\FF_i$ at which we stop. The families $\m{S}_i$ are defined as follows: $\m{S}_i:=\{S_i^{1},\ldots, S_i^{N_i-1}\}$. Clearly, $|S_i^{j}|\le q_i$ for each $j\in [N_i-1]$. The families $\FF_{B}$ promised in (ii) are defined to be $\FF_{i}^j[S_i^j]$ for $B=S_i^j$. Next, note that if $\FF_i^{N_i}$ is non-empty, then 
$$
|\FF_{i}^{N_i}|\le \tau_i^{-|S_i^{N_i}|} \frac{|\Sigma_{n_i}|}{|\Sigma_{n_i}(S)|}|\FF_i(S_i^{N_i})|\le \tau_i^{-|S_i^{N_i}|}|\Sigma_{n_i}|.
$$
We put $\FF_i':=\FF_i^{N_i}$. Since either $|S_i^{N_i}|>q_i$ or $\FF_i' = \emptyset $, we have $|\FF_i'| \leq \tau^{-q_i-1}_i \cdot n_i!$.
\end{proof}

Next, we verify that, thanks to the choice of the  parameters, the families $\m{S}_1$ and $\m{S}_2$ are cross-intersecting. 

\begin{lem}\label{leminters}
    The families $\m{S}_1$ and $\m{S}_2$ (from Theorem \ref{thmapprox}) are cross-intersecting. 
\end{lem}
\begin{proof}
    Assume that there are  $A_1 \in \m{S}_1, A_2 \in \m{S}_2$ such that $A_1 \cap A_2 = \emptyset.$ Recall that $\m{G}_i=\FF_{i, A_i}(A_i)$ are $(\Sigma_{n_i}(B), \tau_i)$-homogeneous. 
    Then
    $$
    |\m{G}_1(\overline{A_2})|\geq |\m{G}_1|-\sum_{x\in A_2\setminus A_1} |\m{G}_1[\{x\}]|\ge \Big(1-\frac {q_2 \tau_1}{n-10\log{n}-1}\Big) |\m{G}_1|\ge \frac{3}{20}|\m{G}_1|,
    $$
    where we used $|\m{G}_1[\{x\}]| \leq \tau_1 \frac{|\Sigma_{n-1}(A_1 \cup \{ x\})|}{|\Sigma_{n-1}(A_1)|}|\m{G}_1| \leq \tau_1 \frac{1}{n-|A_1|-1}|\m{G}_1| \leq \tau_1 \frac{1}{n-10\log{n}-1}|\m{G}_1|$ and $|A_1|\leq q_1 = 10\log{n}, q_2 = \lceil1.1(t-k)\rceil \leq 0.75n$. Analogously, we get that 
$$|\m{G}_2(\overline{A_1})| \geq \Big(1-\frac {q_1 \tau_2}{t-|A_2|-1}\Big) |\m{G}_2|\geq 0.5|\m{G}_2|,
$$ using $|A_2|\leq q_2$, which gives $t-|A_2|-1\geq 1.1k - 0.1t \geq 0.7k$ and $0.7k \geq 2q_1\tau_2\geq2\cdot 10\log{n} \frac{k}{\log^2{n}}$. 
Hence, $\m{G}'_1 :=\m{G}_1(\overline{A_{2}}) \subset {[n]^2\setminus (A_1 \cup A_2) \choose \leq q_1-|A_1|-|A_2|}$ is $(\Sigma_{n-1}, \frac{20}{3}\tau_1)$-homogeneous and $\m{G}'_2:=\m{G}_2(\overline{A_{1}})\subset {[t]^2\setminus (A_1 \cup A_2) \choose \leq q_2-|A_1|-|A_2|}$ is $(\Sigma_{t-1}, 2\tau_2)$-homogeneous, by the last displayed inequality and the trivial inclusion $\m{G}_i(Y, \overline{A_{3-i}}) \subset \m{G}_i(Y)$. This homogeneity implies that for any $Y \subset [n]^2\setminus(A_1\cup A_2)$ we have
$$
\frac{|\m{G}'_1(Y)|}{|\m{G}'_1|} \leq \Big(\frac{20}{3}\tau_1\Big)^{|Y|} \frac{|\Sigma_n(A_1 \cup Y)|}{|\Sigma_n(A_1)|} \leq \Big(\frac{20}{3} \cdot\frac{4\tau_1}{n} \Big)^{|Y|} \leq \Big( \frac{3n}{80\tau_1} \Big)^{-|Y|},
$$
where we used the fact that $\Sigma_n$ is $(\frac{n}{4}, \frac{n}{4})$-spread. Thus, $\m{G}'_1$ is $r_1:=\frac{3n}{80\tau_1}$-spread. Analogously, $\m{G}'_2$ is  $r_2:=\frac{t}{8\tau_2}$-spread.

    We are about to apply the Spread Lemma (Lemma~\ref{spreadtheorem}). For $\m{G}'_1$ we put $m_1 := \log_2{2n}, \delta_1:=\frac{1}{2\log_2{2n}}$, then $m_1\delta_1 = \frac{1}{2}$. For $\m{G}'_2$ we put $m_2 := \log_2{2t}, \delta_2:=\frac{1}{2\log_2{2t}}$, then $m_2\delta_2 = \frac{1}{2}$. Therefore, $r_i\delta_i > 2^{4}$ by our choice of $r_i$ for large $n$.  We randomly color the set $[2,n]^2\setminus(A_1 \cup A_2)$ into $2$ colors. The Spread lemma (modulo simple calculations) implies that for each $i \in \{ 1,2 \}$ a $\frac{1}{2}$-random subset $W$ of $[2,n]^2\setminus(A_1 \cup A_2)$ contains a set from $\m{G}_i$ with probability strictly bigger than $1 - \frac{1}{2}$, i.e., there exists $B_i \in \m{G}_i$  with probability strictly bigger than $1/2$. Using the union bound, we conclude that with positive probability there are such $B_i \in \m{G}_i(\overline{A_{3-i}})$ for each $i \in \{ 1,2 \}$. We get that $(B_1\cup A_1)\cap (B_2\cup A_2)=\emptyset$. This is a contradiction, since $B_i\cup A_i\in \ff_i$.
\end{proof}

Plugging the values of $\tau_i, q_i$ we have the following bounds for $\FF_1', \FF_2'$:
\begin{align*}
& |\FF_1'| \leq 1.01^{-10\log{n}-1}(n-1)! \leq \frac{1}{n^{10}}(n-1)!\leq (n-11)!, \\
& |\FF'_2| \leq \Big(\frac{k}{\log^2{n}}\Big)^{-\lceil 1.1(t-k) \rceil-1}(t-1)!\leq(k-1)!,
\end{align*}
since 
$$
\Big(\frac{k}{\log^2{n}}\Big)^{\lceil 1.1(t-k) \rceil+1} \geq \Big(\frac{k}{\log^2{n}}\Big)^{1.1(t-k)} \geq \Big(\frac{k^{1.1}}{\log^{2.2}{n}}\Big)^{t-k} \geq t^{t-k} \geq \frac{(t-1)!}{(k-1)!},
$$
using $t \leq 4k$ and $k \geq \sqrt{n}/100.$

%When $k +2 \leq t = k + c \leq k + 3k$ we also have
%$$
%|\FF_2'| \leq (\log^{2}{n})^{1.1(t-k)+2}(t-\lceil1.1(t-k)\rceil-1)! \leq (\log^{2}{n})^{1.1(t-k)+2}(1.1k-0.1t)! \leq (k-1)!,
%$$
%that is since 
%$$
%\frac{(\log^{2}{n})^{1.1(t-k)+2}(1.1k-0.1t)!}{(\log^{2}{n})^{1.1(t+1-k)+2}(1.1k-0.1(t+1))!}=\frac{1}{(\log^{2}{n})^{1.1}} \cdot \frac{(k-0.1c)!}{(k-0.1c-0.1)!} \geq 
%$$
%$$\geq \frac{1}{(\log^{2}{n})^{1.1}} \cdot (k - 0.1c)^{0.1} \geq \frac{1}{(\log^{2}{n})^{1.1}} \cdot (0.7k)^{0.1},
%$$
%where we used $k \geq \frac{\sqrt{n}}{100}$ and $c \leq 3k$, hence when $t$ increases the bound on $|\FF'_2|$  can only decrease.

We apply Theorem~\ref{thmapprox}, Lemma~\ref{leminters} and the following calculations to the families $\ff_\Delta \cap \Sigma_{\overline{\alpha_{t+1, n}}}, \ff_\gamma$, playing the roles of $\ff_1,\ff_2$, respectively.
Since $\m{S}_1$ cross-intersects $\m{S}_2$, the family $\widetilde{\FF}_\Delta :=\Big(\FF_\Delta\cap \Sigma_{\overline{\alpha_{t+1, n}}}\Big) \setminus \FF_\Delta' \subset \Sigma[(1,1), \m{S}_1]$ also cross-intersects $\m{S}_2$. We are going to use this to show that $|\widetilde{\FF}_\Delta|$ is small. 
 Note that the approximation $\m S_2$ may be empty, in which case $\widetilde\ff_\Delta$ is empty as well. 

Take $\pi \in \m{S}_2$ of maximal size, and put $|\pi|=\ell$. We may assume that $\pi = \alpha_{t-\ell+1, t}$. Any permutation from  $\widetilde{\FF}_\Delta$ cross-intersects $\pi.$ We also know that for each $(i, i) \in  \alpha_{t-\ell+1, t}$, there is a permutation $\rho_i$ from $\FF_{\gamma}$ that does not contain it. Any permutation from   $\widetilde{\FF}_m$ intersects $\rho_i$ as well. Therefore,   $\widetilde{\FF}_m$ cannot contain any permutation of the form $\Sigma_{\overline{\alpha_{t-\ell+1, n}\setminus\{(i,i)\}}, \overline{\rho_i}}[(1,1), (i,i)]$ for each $i \in  [t-\ell+1, t].$ 
Using Lemma \ref{doubleder}, we have 
\begin{equation}\label{eqgain}
|\Sigma_{\overline{\alpha_{t-\ell+1, n}\setminus\{(i,i)\}}, \overline{\rho_i}}[(1,1), (i,i)]| \geq \frac{(n-2)!}{e^2} - \frac{2(n-3)!}{e^2}.
\end{equation}
Hence, we have
$$
|\widetilde{\FF}_\Delta\cap \Sigma_{\overline{\alpha_{t+1,n}}}[(1, 1)]| \leq |\Sigma_{\alpha_{t-\ell+1, t},\overline{\alpha_{t+1, n}}}[(1,1)]|-\ell\frac{(n-2)!}{e^2} + \ell\frac{2(n-3)!}{e^2}=:X.
$$

Recall that $$
|\Sigma_{\alpha_{t-\ell+1, t},\overline{\alpha_{t+1, n}}}[(1,1)]| = \sum_{i=0}^{t-1} {t-1 \choose i}\frac{(n-i-1)!}{e} - \sum_{i=0}^{t-\ell-1} {t-\ell-1 \choose i}\frac{(n-i-1)!}{e}.
$$

We compare $X$ with the size of $\m{N}(\HH_k)\cap \Sigma_{\overline{\alpha_{t+1,n}}}[(1, 1)]$, given by the formula $$|\m{N}(\HH_k)\cap \Sigma_{\overline{\alpha_{t+1,n}}}[(1, 1)]|=\sum_{i = 0}^{t-1} {t-1 \choose i} \frac{(n-i-1)!}{e} - \sum_{i = 0}^{k-1} {k-1 \choose i}\frac{(n-i-1)!}{e}.$$

First, assume that $\ell \geq t-k$.
We compare the difference between $|\m{N}(\HH_k)\cap \Sigma_{\overline{\alpha_{t+1,n}}}[(1, 1)]|$ and $|\Sigma_{\alpha_{t-\ell+1, t},\overline{\alpha_{t+1, n}}}[(1,1)]|$:
$$
\sum_{i = 0}^{k-1} {k-1 \choose i}\frac{(n-i-1)!}{e} - \sum_{i=0}^{t-\ell-1} {t-\ell-1 \choose i}\frac{(n-i-1)!}{e}
$$
$$
\leq \sum_{i=1}^{k-1} (k-t+l){k-1 \choose i-1}\frac{(n-i-1)!}{e} + O((n-3)!)\leq  (k-t+\ell)\sum_{i=0}^{\infty}\frac{1}{i!} \frac{(n-2)!}{e}= (k-t+\ell) (n-2)!,
$$
where we used ${n \choose k} - {n-m \choose k}\leq m{n-1 \choose k-1},$ valid for any $n\ge k,m$, and in the first inequality we used $t-\ell-1 > (k-1)/2$, meaning that the remaining tail is in the $O((n-3)!)$. Using this calculation, we get
\begin{align*}|\m{N}(\HH_k)\cap \Sigma_{\overline{\alpha_{t+1,n}}}[(1, 1)]|-X &\ge \ell\frac{(n-2)!}{e^2} - \ell\frac{2(n-3)!}{e^2} - (k-t+\ell) (n-2)!\\
&\geq \left (\frac{\ell}{e} - e(k-t+\ell) -O((t-k)/n)\right )\frac{(n-2)!}{e} \\
& \geq \left( \ell(1/e - e) -e(k-t) - O((t-k/n)) \right)\frac{(n-2)!}{e}\\
& \geq \left (\frac{\lceil 1.1(t-k) \rceil}{e} - e\lceil 0.1(t-k)\rceil -O((t-k)/n)\right )\frac{(n-2)!}{e}, 
\end{align*}
where we used $\ell\leq \lceil 1.1(t-k) \rceil $.  This is at least $1.2(n-2)!$ for $t \geq k + 35.$

Second, assume that $\ell \leq t-k-1$. In this case, we get that
\begin{align*}
    |\m{N}(\HH_k)\cap \Sigma_{\overline{\alpha_{t+1,n}}}[(1, 1)]|-X & \leq \sum_{i=0}^{t-\ell-1} {t-\ell-1 \choose i}\frac{(n-i-1)!}{e} - \sum_{i = 0}^{k-1} {k-1 \choose i}\frac{(n-i-1)!}{e} + l\frac{(n-2)!}{e^2}\\
    &  \le (t-k-\ell)\frac{(n-2)!}{e}+\ell\frac{(n-2)!}{e^2}.
    & .
\end{align*}
This is also at least, say, $1.2(n-2)!$ for $t \geq k + 35.$

Hence, for $t \geq k + 35$ we have 
$$
|\m{E}_k|-|\FF | \geq  |\m{N}(\HH_k)\cap \Sigma_{\overline{\alpha_{t+1,n}}}[(1, 1)]|-X -|\ff'_\Delta|-|\ff_\gamma|\ge 1.2(n-2)!-|\ff'_\Delta|-|\FF_\gamma|>0,
$$
since $|\FF_\gamma| \leq (n-2)! - (n-3)!$ and $|\ff'_\Delta|=|\ff'_1| \leq (n-11)!$. The theorem is proved in this case. Note that, we have the following bound on $\FF$:
$$
|\FF|\leq|\m{E}_k|-1.2(n-2)!+|\FF'_\Delta|+|\FF_\gamma|.
$$

Next, we deal with the case $t \leq k + 35$. Consider $\sigma:=\bigcap \m{S}_2, |\sigma|=\ell'$. W.l.o.g. $\sigma = \alpha_{t-\ell'+1, t}$.  Then $\widetilde{\ff}_{\gamma}:=\ff_\gamma\setminus \ff'_\gamma\subset \Sigma [\alpha_{t-\ell'+1, n}]$. At the same time, $|\widetilde{\ff}_\gamma|\ge k!-2(k-1)!>(k-1)!$ by our bound on the remainder. It means that $(t-\ell') > (k-1)$, hence $\ell' \leq t-k$. Consider $$\m{S}_2(\sigma)=\{ \pi \setminus \sigma \mid \pi \in \m{S}_2 \} \subset {[2, t]^2 \setminus \sigma \choose \le 1.1(t-k) - \ell'}.$$ Note that it may be empty.

If $\m S_2(\sigma)$ is non-empty, then for any $x \in [2, t]^2 \setminus \sigma$ there is a partial permutation $\rho_x\in \m S_2(\sigma)$ such that $x \notin \rho_x$.  By the same argument as in the Lemma \ref{hitsets} we have that the number of inclusion-minimal $i$-hitting sets of $\m{S}_2(\sigma)$ is at most $(1.1(t-k)-|\sigma|)^i \leq 38^i$ for $i\geq 2$.  Then $\widetilde{\FF}_\Delta$ cross-intersects $\m{S}_2$ either by cross-intersecting its common part $\sigma$, or by cross-intersecting $\m{S}_2(\sigma)$ . We  bound the former part as in Case 2 by plugging $38$ instead of $t$ in \ref{eqcompare3}: $$38^2\frac{(n-3)!}{e}+o((n-3)!)=O((n-3)!).$$ 
We could do the same calculations as above to bound the latter part, but now plugging the value $\ell'=t-k$ in the bounds above for the case $t \geq k + 35$ (it is easy to see that this value of $\ell'$ gives the worst bound). It actually makes $\m{N}(\HH_k)\cap \Sigma_{\overline{\alpha_{t+1,n}}}[(1, 1)]$ and $\Sigma_{\alpha_{t-\ell'+1, t},\overline{\alpha_{t+1, n}}}[(1,1)]$ coincide, we have
$$
|\FF| \leq |\Sigma_{n, \alpha_{k+1, n}}[(1,1)]|-(t-k)\frac{(n-2)!}{e^2}+(t-k)\frac{2(n-2)!}{e^2}+O((n-3)!) + |\FF'_{\Delta}| + |\FF_\gamma|, 
$$
which gives
$$
|\m{E}_k|-|\FF | \geq  (t-k)\frac{(n-2)!}{e^2}- (t-k)\frac{2(n-3)!}{e^2}-O((n-3)!)+k!-(k-1)!-|\ff'_\Delta|-|\ff'_\gamma|-|\widetilde\ff_\gamma| >0,
$$
when $|\widetilde\FF_\gamma| < \frac{(n-2)!}{100}$.

%Now let us look 

%or $k = n-2.$ \textcolor{red}{our theorem normally does not allow $k=n-2$, only $k=n-3,$ no?}

Hence, the remaining case is $|\widetilde\FF_\gamma| \geq \frac{(n-2)!}{100}$. Since $\widetilde{\ff}_{\gamma}\subset \Sigma [\alpha_{t-\ell'+1, n}]$, we must have 
$t-\ell'+1\ge n-1$.  Since $\m E_k$, $k=2,\ldots, n-3$ is the smallest for $k=n-3,$ we may w.l.o.g. compare $|\ff|$ with $\m E_{n-3}.$
Recall that, using the inclusion-exclusion formula, we have
$$
|\m{E}_{n-3}|=\big(3(n-2)!-3(n-3)!+(n-4)!\big)+(n-3)!-(n-4)!=3(n-2)!-2(n-3)!
$$
We bound $|\FF|$ using \eqref{eqgain} as follows
\begin{align*}
|\FF| &\leq |\Sigma_{\alpha_{t-\ell'+1, n}} [(1,1)]| - |\sigma|\left(\frac{(n-2)!}{e^2}-\frac{(n-3)!}{e^2} \right)+ |\ff'_\Delta|+|\FF_{\gamma}|\\
&\leq \left ((n-t+\ell')(n-2)!-{n-t+\ell' \choose 2}(n-3)!\right)- \ell'\frac{(n-2)!}{e^2} +(n-2)!+ O((n-3)!)\\
&\leq \big(n-t+\ell'+1-e^{-2}\ell'\big)(n-2)!+ O((n-3)!)
\end{align*}
where in the second inequality we  used the inclusion-exclusion formula for $|\Sigma_{\alpha_{t-\ell'+1, n}} [(1,1)]|$, the fact that $n-t+\ell'\le 2$ and the fact that $|\FF_{\gamma}| \leq (n-2)!-(n-3)!$.
Since $n-t+\ell'\le 2$, we have $n-t+1+\ell'(1-e^{-2})\le 3-e^{-2}$  unless $t=n-2$ and $\ell'=0$. Unless the latter happens, we have $|\ff|<|\m E_{n-3}|$ by comparing the coefficients in front of $(n-2)!$ and, third, $\ell' = 0.$ 

We are left with the last case when, first, $t=n-2,$ that is, all permutations in $\ff_\gamma$ have two common elements $(n-1,n-1)$ and $(n,n)$, second, $|\ff_\gamma|\ge (n-2)!/100$. We may actually assume that $$|\ff_\gamma|\ge \frac {99(n-2)!}{100},$$
since otherwise, substituting this improved bound on $|\ff_\gamma|$ in the calculation above, we again get 
$$|\ff| =
\big(n-t+\ell'+99/100-e^{-2}\ell'\big)(n-2)!+ O((n-3)!) = \big(2+99/100\big)(n-2)!+ O((n-3)!) <|\m E_{n-3}|.
$$ We claim that then all permutations from $\ff_\Delta$ must contain either $(n-1,n-1)$ or $(n,n)$. This will conclude the proof, since then $\ff\subset \m E_{n-2},$ and we know that $|\m E_{n-2}| = |\m E_{n-3}|.$ 

Assume that there is a permutation $\pi\in \ff_\Delta$ such that $\pi(i)\ne i$ for $i=n-1,n.$ The number of permutations from $\Sigma[\overline{(1,1)}, \alpha_{n-1,n}]$ that do not intersect $\pi$ is asymptotically the number of derangements on a set of size $n-2$, that is, $(1+o(1))(n-2)!/e$. All these permutations must be excluded from $\ff_\gamma$, which implies that $|\ff_\gamma|\le (1-e^{-1}+o(1))(n-2)!$, a contradiction with the last displayed bound. This concludes the proof of the theorem.

%We have $k \leq n-3$, and thus. Then $t \geq n-2$. In this case $ \frac{(n-2)!}{100} - (n-3)!\leq|\widetilde{\FF}_{\gamma}|\leq|\widetilde{\FF}_{\gamma}[\sigma]|$. However, if $|\sigma|=\ell'\geq t-k$, then $|\widetilde{\FF}_{\gamma}[\sigma]| \leq k!\leq(n-3)!$. So we have $\ell' \leq t-k-1$. Using the inclusion-exclusion formula, we have
%$$
%|\m{E}_k|=\left ((n-k)(n-2)!-{n-k \choose 2}(n-3)!+\dots\right )+k!-(k-1)!,
%$$
%where $n-k \leq 37$, so, in particular, a constant number of terms. 

%Finally, since $t-l\geq  k+1$, we have 
%$$
%|\m{E}_k| -|\FF| \geq \frac{(n-2)!}{e^2} + O((n-3)!).
%$$

 %The theorem is proved.    

\end{document}